\documentclass{article}
\usepackage{graphicx} 

\usepackage{hyperref}
\usepackage{amsmath}
\usepackage{amsthm}
\usepackage{amssymb}
\usepackage{amsfonts}
\usepackage{multicol}
\usepackage{subfigure}
\usepackage[numbers,sort&compress]{natbib}
\usepackage{booktabs}
\usepackage{longtable,tabu}

\usepackage{enumitem}
\setlist[enumerate]{label={\arabic*.}}

\usepackage{color}

\newcommand{\infcrit}[2]{G(#1,#2)}
\newcommand{\critset}[2]{\operatorname{crit}_{#1}(#2)}
\newcommand{\numcrit}[2]{\left\vert\operatorname{crit}_{#1}\left(#2\right)\right\vert}

\newcommand{\kcol}{$k$-\textsc{Colouring}}

\usepackage{tikz}
\usetikzlibrary{calc}

\usepackage{tkz-graph}

 \newtheorem{theorem}{Theorem}[section]
 \newtheorem{lemma}[theorem]{Lemma}
 
  \newtheorem{observation}[theorem]{Observation}
 
 \newtheorem{corollary}[theorem]{Corollary}

 \newtheorem{definition}[theorem]{Definition}
 
 \newtheorem{conjecture}[theorem]{Conjecture}
 \newtheorem{problem}{Problem}

\title{A dichotomy for the number of vertex-critical ($P_5$, $H$)-free graphs when $H$ is bipartite}
\author{Iain Beaton\\ 
\small Department of Mathematics and Statistics\\
\small Acadia University\\
\small Wolfville, NS Canada\\
\small iain.beaton@acadiau.ca\\
\and
Ben Cameron\\ 
\small School of Mathematical and Computational Sciences\\
\small University of Prince Edward Island\\
\small Charlottetown, PE Canada\\
\small brcameron@upei.ca\\
}
\date{\today}

\begin{document}

\maketitle

\begin{abstract}
    A graph $G$ is $k$-vertex-critical if $\chi(G)=k$, but $\chi(H)<k$ for every induced subgraph $H$ of $G$. A graph $G$ is $(H_1,H_2,\dots,H_m)$-free if does not contain $H_i$ as an induced subgraph for any $i\in\{1,2,\dots,m\}$.

    We provide the following dichotomy that for bipartite graphs $H$ and any fixed integer $k\ge 5$, there are only finitely many $k$-vertex-critical $(P_5,H)$-free graphs if and only if $H$ is $2P_2$-free.
    This leads us to pose the problem about determining for which graphs $H$ with $\chi(H)\ge 3$ there are infinitely many $k$-vertex-critical $(P_5,H)$-free graphs for all $k\ge 5$.  Toward this problem, we show that there only finitely many $k$-vertex-critical $(P_5, K_{s,t}+e)$-free graphs for all $k,s,t\ge 1$, where $K_{s,t}+e$ is a complete bipartite graph plus a single edge.
    On the other hand, we show that there are infinitely many $k$-vertex-critical $(P_5,\operatorname{net},\operatorname{co-net},\overline{C_5},\overline{C_6},\dots\overline{C_{k-1}})$-free graphs for all $k\ge 5$. We also show that there are only finitely many $k$-vertex-critical $(P_4+\ell P_1,\overline{L(K_{2,n})})$-free graphs for all $\ell,n\ge 0$, providing the largest known subfamily of $(P_4+\ell P_1)$-free graphs to satisfy this property.
    Our results, together with known results, imply the existence of new polynomial-time certifying algorithms to determine the $k$-colourability of many subfamilies of $P_5$-free and $(P_4+\ell P_1)$-free graphs for fixed $k\ge 5$.
    
    Our proof techniques apply a powerful theorem of  Chudnovsky, Kim, Oum, and Seymour (2016) on prime graphs that we expect to be of interest and have further applications to bounding the number of $k$-vertex-critical graphs in other hereditary families of graphs.
\end{abstract}

\section{Introduction}\label{sec:intro}

A graph is $k$-vertex-critical if $\chi(G)=k$ but every induced subgraph of $G$ is $(k-1)$-colourable.
Let \kcol{} denote the decision problem of determining if a given input graph is $k$-colourable.
Since every graph that is not $k$-colourable necessarily contains a $(k+1)$-vertex-critical induced subgraph, a $(k+1)$-vertex-critical graph can be returned with each negative output of \kcol{} to certify that it is correct.
McConnell, Mehlhorn, N\"{a}her, and Schweitzer~\cite{McConnell2011} make convincing arguments that certifying algorithms are highly desirable given the ease at which they can be robustly tested; however, the certificate must be easily (i.e., efficiently) verifiable to provide value.
Since there is no known efficient algorithm to decide if an input graph is $(k+1)$-vertex-critical for all $k\ge 3$, using $(k+1)$-vertex-critical induced subgraphs as ``no''-certificates for \kcol{} is not practical in general.
However, the structure of specific graph classes can allow for the efficient decidability of $k$-vertex-critical graphs in the family for fixed $k$, as we describe in the next paragraph.

A hereditary class of graphs is a class of graphs that is closed under taking induced subgraphs.
It is well-known that every class of graphs is defined by a set of forbidden induced subgraphs if and only if it is a hereditary class or graphs.
We say a graph is $H$-free if it does not contain $H$ as an induced subgraph, and more generally, a graph is $(H_1,H_2,\dots, H_n)$-free if it does not contain $H_i$ as induced subgraph for all $i\in \{1,2,\dots,n\}$.
Bruce, Ho\`{a}ng, and Sawada~\cite{Bruce2009} initiated the study of critical graphs in hereditary graph classes and their work was later extended by Maffray and Morel~\cite{MaffrayMorel2012} to show that there are exactly twelve $4$-vertex-critical $P_5$-free graphs.
This result gives a polynomial-time algorithm to solve $3$-\textsc{Colouring} by simply searching for each of the twelve $4$-vertex-critical graphs as induced subgraphs of the input graph. Further, if one is found as an induced subgraph, it can be returned to efficiently verify the negative outputs by searching for it among the twelve.
More generally, if a hereditary class of graphs contains only finitely many $(k+1)$-vertex-critical graphs, then there is a polynomial-time algorithm that provides ``no''-certificates to solve \kcol{} when the inputs are restricted graphs from to this graph class (see~\cite{P5banner2019}, for example, for the full details).
This connection and the fact that \kcol{} is NP-complete in general for all $k\ge 3$~\cite{Karp1972} has led to an explosion of interest into proving various hereditary families of graphs contain only finitely many $k$-vertex-critical graphs for some $k\ge 4$. 
For a positive integer $k$ and graphs $H_1$, $H_2$, $\dots,\ H_m$, let $\critset{k}{H_1, H_2, \dots,\ H_m}$ denote the set of all $k$-vertex-critical $(H_1, H_2, \dots,\ H_m)$-free graphs. 
Now $\numcrit{k}{H_1, H_2, \dots,\ H_m}$ denotes the number of $k$-vertex-critical $(H_1, H_2, \dots,\ H_m)$-free graphs.
Thus, when $\numcrit{k}{H_1, H_2, \dots,\ H_m}<\infty$, there is a polynomial-time \kcol{} algorithm that provides ``no''-certificates for the class of $(H_1, H_2, \dots,\ H_m)$-free graphs.

The most comprehensive result in this area is the dichotomy theorem of Chudnovksy, Goedgebeur, Schaudt, and Zhong~\cite{Chud4critical2020} that $\numcrit{4}{H}<\infty$ if and only if $H$ is an induced subgraph of $P_6$, $2P_3$, or $P_4+\ell P_1$ for any $\ell\ge 0$. 
A natural question that arises from this result is what happens for large values of $k$?
When $H$ contains an induced claw (see, for example,~\cite{CameronHoangSawada2022}) or cycle (follows from~\cite{Erdos}), then for all $k\ge 3$, $\numcrit{k}{H}=\infty$.
Ho\`{a}ng, Moore, Recoskie, Sawada, and Vashtelle~\cite{Hoang2015} also showed that $\numcrit{k}{2P_2}=\infty$ for all $k\ge 5$.
Thus, graphs $H$ such that $\numcrit{k}{H}$ for some $k\ge 4$ must be linear forests (that is, forests where every component is a path) with one component having four or fewer vertices and the remaining components all singletons.
These are exactly the induced subgraphs of $P_4+\ell P_1$.
In light of this, the second author, Ho\`{a}ng, and Sawada~\cite{CameronHoangSawada2022} posed the following conjecture that is one of the main motivations for the results in this paper.

\begin{conjecture}[\cite{CameronHoangSawada2022}]\label{conj:P4UellP1finite}
    Let $k\ge 5$. There are only finitely many $k$-vertex-critical $H$-free graphs if and only if $H$ is an induced subgraph of $P_4+\ell P_1$ for some $\ell\ge 0$.
\end{conjecture}

\noindent This conjecture was supported by their result that $\numcrit{k}{P_2+\ell P_1}<\infty$ for all natural numbers $k$ and $\ell$.  
Later, Abuadas, the second author, Ho\`{a}ng, and Sawada~\cite{AbuadasCameronHoangSawada2024} improved this result by showing that $\numcrit{k}{P_3+\ell P_1}<\infty$ for all natural numbers $k$ and $\ell$. 
Surprisingly, Conjecture~\ref{conj:P4UellP1finite} remains open even for the most restrictive non-trivial case where $k=5$ and $\ell=1$.
Recently, Beaton and Cameron~\cite{BeatonCameron2026IWOCA} studied the $k$-vertex-critical graphs in the more restrictive class of $(P_4+\ell P_1,H)$-free, showing that there are only finitely many such graphs for all natural numbers $k$ and $\ell$ when $H=2P_2$ or chair, among other results.
Very recently, Belavadi and Karthick~\cite{KarthickBelavadi2026} showed that $\numcrit{k}{P_4+P_1,H}<\infty$ for all natural numbers $k$ when $H$ is any of the graphs of order $5$ from the set $\{\text{paraglider, dart, house}\}$.
Their work was motivated by an open question of the authors~\cite{BeatonCameron2025cogemfreeord4finite} that arose after we showed $\numcrit{k}{P_4+P_1,H}<\infty$ for all natural numbers $k$ and graphs $H$ of order $4$.   
This result was in analogy to the result of K. Cameron, Goedgebeur, Huang, and Shi~\cite{KCameron2021}  that $\numcrit{k}{P_5,H}<\infty$ for all natural numbers $k$ and graphs $H$ of order four if and only if $H$ is not $2P_2$ or $K_3+P_1$.
K. Cameron et al. also posed the following open question that provides the other main motivation for our results in this paper.

\begin{problem}[\cite{KCameron2021}]\label{prob:P5Hord5}
    For which graphs $H$ of order $5$ are there only finitely many $k$-vertex-critical $(P_5,H)$-free graphs for $k\ge 5$?
\end{problem}

It should be noted that \kcol{} $P_5$-free graphs is polynomial-time solvable for all natural number $k$~\cite{Hoang2010}, and $P_5$ is the largest connected graph that can be forbidden where this occurs~\cite{Huang2016} (assuming P$\neq$NP).
Further, the polynomial-time \kcol{} algorithms for $P_5$-free graphs return $k$-colourings to certify ``yes''-output.
Thus, determining exactly which subfamilies of $P_5$-free graphs admit fully certifying, efficient \kcol{} algorithms is of considerable interest.
Much of this recent interest has focused on Problem~\ref{prob:P5Hord5} where it is now known to be true for the following notable graphs:

\begin{multicols}{2}
\begin{itemize}
 \item banner~\cite{Brause2022}
 \item $K_{2,3}$ or $K_{1,4}$~\cite{Kaminski2019}
  \item chair or cricket~\cite{Jooken2026}
 \item $\overline{P_5}$~\cite{Dhaliwal2017}
 \item $\overline{P_3+P_2}$ or gem~\cite{CaiGoedgebeurHuang2023} 
 \item dart~\cite{Xiaetal2023}
 \item $K_{1,3}+P_1$ or $\overline{K_3+2P_1}$~\cite{xia2024results}
 \item $W_4$~\cite{P5W4Journal}
 \item bull~\cite{BelavadiHoang2026}
\end{itemize}
\end{multicols}

The only new infinite families of $k$-vertex-critical $(P_5,H)$-free graphs discovered since Problem~\ref{prob:P5Hord5} was posed are due to the second author and Ho\`{a}ng who provided one for $H=C_5$ and each $k\ge 6$~\cite{CameronHoang2023P5C5}.
This is in contrast to the fact that there are only finitely many $5$-vertex-critical $(P_5,C_5)$-free graphs~\cite{Hoang2015}. 
Combining these results and others stated above, we find that Problem~\ref{prob:P5Hord5} now remains open only when $H$ is one of the five graphs in Figure~\ref{fig:graphsoforder5}.
From this (i.e., Figure~\ref{subfig:cogem}), it is clear that there is overlap in interest in solving Conjecture~\ref{conj:P4UellP1finite} and Problem~\ref{prob:P5Hord5}. 

\begin{center}
\begin{figure}[htb]
\def\c{0.4}
\def\r{1.5}
\centering
\qquad
\subfigure[$P_4+P_1$]{
\scalebox{\c}{
\begin{tikzpicture}
\begin{scope}[every node/.style={circle,fill,draw}]
    \node (u1) at (-1*\r,0*\r) {};
    \node (u2) at (1*\r,0*\r) {};
    \node (u3) at (0*\r,-1*\r) {};
    \node (u4) at (-1*\r,-2*\r) {};
    \node (u5) at (1*\r,-2*\r) {};    
\end{scope}
\begin{scope}
    \path [-] (u1) edge node {} (u2);    
    \path [-] (u5) edge node {} (u2);
    \path [-] (u4) edge node {} (u1);       
\end{scope}
\end{tikzpicture}}
\label{subfig:cogem}
}
\qquad
\subfigure[$C_4+P_1$]{
\scalebox{\c}{
\begin{tikzpicture}
\begin{scope}[every node/.style={circle,fill,draw}]
    \node (u1) at (-1*\r,0*\r) {};
    \node (u2) at (1*\r,0*\r) {};
    \node (u3) at (0*\r,-1*\r) {};
    \node (u4) at (-1*\r,-2*\r) {};
    \node (u5) at (1*\r,-2*\r) {};    
\end{scope}
\begin{scope}
    \path [-] (u1) edge node {} (u2);    
    \path [-] (u5) edge node {} (u2);
    \path [-] (u4) edge node {} (u1);    
    \path [-] (u4) edge node {} (u5);    
\end{scope}
\end{tikzpicture}}
\label{subfig:C4+P1}
}
\qquad
\subfigure[$\overline{P_3+2P_1}$]{
\scalebox{\c}{
\begin{tikzpicture}
\begin{scope}[every node/.style={circle,fill,draw}]
    \node (u1) at (-1*\r,0*\r) {};
    \node (u2) at (1*\r,0*\r) {};
    \node (u3) at (0*\r,1*\r) {};
    \node (u4) at (-1*\r,-2*\r) {};
    \node (u5) at (1*\r,-2*\r) {};    
\end{scope}
\begin{scope} 
    \path [-] (u1) edge node {} (u3);  
    \path [-] (u2) edge node {} (u3); 
    \path [-] (u1) edge node {} (u2);  
    \path [-] (u1) edge node {} (u4);
    \path [-] (u1) edge node {} (u5);  
    \path [-] (u5) edge node {} (u2);
    \path [-] (u4) edge node {} (u2);   
    \path [-] (u4) edge node {} (u5);    
\end{scope}
\end{tikzpicture}}
\label{subfig:crosshouse}
}
\qquad
\subfigure[$K_5-e$]{
\scalebox{\c}{
\begin{tikzpicture}
\begin{scope}[every node/.style={circle,fill,draw}]
    \node (u1) at (-1*\r,0*\r) {};
    \node (u2) at (1*\r,0*\r) {};
    \node (u3) at (0*\r,1*\r) {};
    \node (u4) at (-1*\r,-2*\r) {};
    \node (u5) at (1*\r,-2*\r) {};    
\end{scope}
\begin{scope} 
    \path [-] (u1) edge node {} (u3);  
    \path [-] (u2) edge node {} (u3); 
    \path [-] (u4) edge node {} (u3); 
    \path [-] (u1) edge node {} (u2);  
    \path [-] (u1) edge node {} (u4);
    \path [-] (u1) edge node {} (u5);  
    \path [-] (u5) edge node {} (u2);
    \path [-] (u4) edge node {} (u2);   
    \path [-] (u4) edge node {} (u5);    
\end{scope}
\end{tikzpicture}}
\label{subfig:K5-e}
}
\qquad
\subfigure[$K_5$]{
\scalebox{\c}{
\begin{tikzpicture}
\begin{scope}[every node/.style={circle,fill,draw}]
    \node (u1) at (-1*\r,0*\r) {};
    \node (u2) at (1*\r,0*\r) {};
    \node (u3) at (0*\r,1*\r) {};
    \node (u4) at (-1*\r,-2*\r) {};
    \node (u5) at (1*\r,-2*\r) {};    
\end{scope}
\begin{scope} 
    \path [-] (u1) edge node {} (u3);  
    \path [-] (u2) edge node {} (u3); 
    \path [-] (u1) edge node {} (u2);  
    \path [-] (u1) edge node {} (u4);
    \path [-] (u1) edge node {} (u5);  
    \path [-] (u5) edge node {} (u2);
    \path [-] (u4) edge node {} (u2);   
    \path [-] (u4) edge node {} (u5); 
    \path [-] (u4) edge node {} (u3); 
    \path [-] (u5) edge node {} (u3);    
\end{scope}
\end{tikzpicture}}
\label{subfig:cK5}
}
\caption{Graphs $H$ of order $5$ where the finiteness of $k$-critical $(P_5,H)$-free graphs is unknown.}\label{fig:graphsoforder5}
\end{figure}
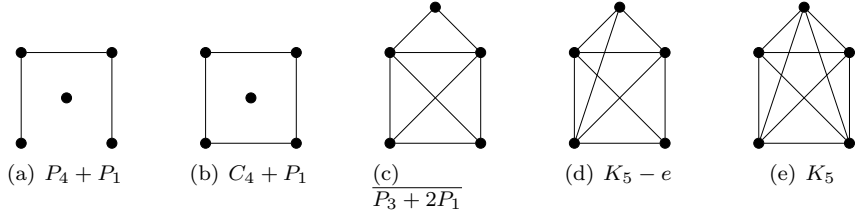
\end{center}

\subsection{Our contributions}
Corollaries of our main results are that $\numcrit{k}{P_5,H}<\infty$ for all $k$ when $H$ is $P_4+\ell P_1$ or $C_4+\ell P_1$ for any $\ell\ge 0$, thus solving two of the five outstanding remaining open cases of Problem~\ref{prob:P5Hord5} and providing more evidence for the validity of Conjecture~\ref{conj:P4UellP1finite}. 
In fact, these cases are very specific corollaries of one of our main results (Theorem~\ref{thm:P5Bipartite}) that for a bipartite graph $H$ and any $k\ge 5$, $\numcrit{k}{P_5,H}<\infty$ if and only if $H$ is $2P_2$-free.
This is the largest subfamily of $P_5$-free graphs where the number of $k$-vertex-critical graphs in the family is shown to be finite and we believe this result helps illuminate the direction forward for the study of $P_5$-free vertex-critical graphs.
Instead of forbidding graphs in addition to $P_5$ incrementally by order, it may be more natural to consider forbidding graphs incrementally by chromatic number. 
This motivates the following open problem.

\begin{problem}\label{prob:P5Hlargerchrom}
    For which graphs $H$ with $\chi(H)\ge 3$ is $\numcrit{k}{P_5,H}<\infty$ for all natural numbers $k$?
\end{problem}

Note that if $\chi(H)=k$ and $k\ge 5$, then we have $\numcrit{k}{P_5,H}=\infty$ as at most one graph from the infinite family from~\cite{CameronHoang2023P5C5} is forbidden.
Thus, to solve Problem~\ref{prob:P5Hlargerchrom} all that remains is to consider graphs $H$ with $3\le \chi(H)\le 4$.
Note that all of the remaining cases of Problem~\ref{prob:P5Hord5} have chromatic number at least 4, and are indeed the only $(2P_2,K_3+P_1)$-free graphs of order 5 that are not $3$-colourable. 
In fact, the only graph $4$-chromatic graph $H$ where it is known that $\numcrit{k}{P_5,H}<\infty$ graphs for some $k\ge 5$ is $H=K_4$, where it is known for all $k\ge 5$~\cite{KCameron2021}.

Toward a solution to Problem~\ref{prob:P5Hlargerchrom} when $\chi(H)=3$, we show that for all $k, s\geq 1$ and $t \geq 2$, $\numcrit{k}{P_5, K_{s,t}+e}<\infty$ where $K_{s,t}+e$ is a complete bipartite graph with an edge added between two of the $t$ vertices in one set of the bipartition.
On the other hand, we also prove new structure for the known infinite families of $k$-vertex-critical $P_5$-free graphs from~\cite{Hoang2015,CameronHoang2023P5C5} to show that $\numcrit{k}{P_5,H}=\infty$ for all $k\ge 5$ when $H$ is net or co-net for all $k$.

Of course, for graphs $H$ with chromatic number $k \ge 5$, it is still of interest to determine when $\numcrit{k'}{P_5,H}<\infty$ for all $k'>k$.
To this end, we also show that $\numcrit{k}{P_5,H}=\infty$ when $H$ is $\overline{C_{m}}$ for all $m\le k-1$.
Thus, from the Strong Perfect Graph Theorem~\cite{Chudnovsky2006}, it follows that any graph $H$ where Problem~\ref{prob:P5Hlargerchrom} holds must be perfect.

Not to be lost among our work on Problem~\ref{prob:P5Hord5} is the progress we make on Conjecture~\ref{conj:P4UellP1finite} by showing that $\numcrit{k}{P_4+\ell P_1,\overline{L(K_{2,n})}}<\infty$ for all natural numbers $k, \ell,$ and $n$.

Our results together with those in~\cite{Hoang2010} and~\cite{Couturier2015} imply the existence of many new polynomial-time certifying \kcol{} algorithms for subfamilies of $P_5$-free graphs and subfamilies of $(P_4+\ell P_1)$-free graphs and generalize many existing results on both $P_5$-free and $(P_4+\ell P_1)$-free vertex-critical graphs.

\subsection{Outline}

We begin with definitions, notations, helpful known results on vertex-critical graphs in Section~\ref{sec:prelims}.
We then use a strong theorem of Chudnovsky, Kim, Oum, and Seymour~\cite{CHUDNOVSKY20161} on unavoidable induced subgraphs in large prime graphs to show the following in Section~\ref{sec:unavoidablestructures}:
\begin{itemize}
    \item $\numcrit{k}{P_5,H_n}<\infty$ or all $k,n\ge 1$.
    \item $\numcrit{k}{P_4+\ell P_1,\overline{L(K_{2,n}
    }}<\infty$ for all $k\ge 1$, $\ell\ge 0$,.   
\end{itemize}
where $\overline{L(K_{2,n})}$ is the complement of the line graph of $K_{2,n}$ and $H_n$ is the so-called half-graph defined in Section~\ref{sec:unavoidablestructures} and shown in Figure~\ref{subfig:halfgraph5labelled}.
From these, it follows immediately that $\numcrit{k}{P_5,P_4+\ell P_1}\infty$ and $\numcrit{k}{P_5,C_4+\ell P_1}\infty$ for all $k\ge 1$, $\ell\ge 0$, with the $\ell=1$ case resolving two outstanding cases of Problem~\ref{prob:P5Hord5}.
Further, show that for a fixed bipartite graph $H$ and integer $k\ge 5$, $\numcrit{k}{P_5,H} <\infty$ if and only if $H$ is $2P_2$-free.
In Section~\ref{sec:P5HnonbipartiteH}, we then look to $(P_5,H)$-free graphs for non-bipartite graphs $H$ and, building on our result from Section~\ref{sec:unavoidablestructures}, show that for all $k, s\geq 1$ and $t \geq 2$, $\numcrit{k}{P_5, K_{s,t}+e}<\infty$.
Also in Section~\ref{sec:P5HnonbipartiteH}, we prove further structural results on the infinite families constructed in~\cite{Hoang2015,CameronHoang2023P5C5} to show that $\numcrit{k}{P_5,H}=\infty$ for:
\begin{itemize}
    \item $H=$ net and $k\ge 5$,
    \item $H=$ co-net and $k\ge 5$, and
    \item $H=\overline{C_m}$ and $k\ge m+1$.
\end{itemize}
We then conclude in Section~\ref{sec:conclusion} with directions for future research and a few results on $\numcrit{5}{P_5,H}$ when $H$ contains an induced $C_5$.

\section{Preliminaries}\label{sec:prelims}
In addition to the definitions and notation introduced in Section~\ref{sec:intro}, we provide more here to aid the reader and improve the flow of the remaining sections.
This list is not exhaustive, and when definitions or notation are undefined, we follow the standard graph theory definitions and notations in~\cite{WestGraphTheoryBook}.
Unless otherwise specified, we use the naming conventions of \url{https://graphclasses.org/smallgraphs.html} for named small graphs.
Note, however, that we use $G+H$ to denote the disjoint union of graphs $G$ and $H$, and $mG$ to denote the disjoint union of $m$ copies of $G$ when $m$ is a non-negative integer.
We write $x\sim y$ to denote vertices $x$ and $y$ being adjacent and $x\nsim y$ to denote that $x$ is not adjacent to $y$.
For a graph $G$ and sets $X,Y\subseteq V(G)$, we say that $X$ is \textit{complete} (respectively, \textit{anticomplete}) if $x\sim y$ (respectively, $x\nsim y$) for all $x\in X$.
If $X=\{x\}$ and $X$ is (anti)complete to $Y$, we say that the vertex $x$ is (anti)complete to $Y$.
A set $X\subseteq V(G)$ is a \textit{homogeneous set} if every vertex $y\in V(G)\setminus X$ is complete or anticomplete to $X$.
A homogeneous set $X$ of a graph $G$ is said to be a \textit{nontrivial homogeneous set} if $1 < |X| < |V(G)|$.
A graph is called \textit{prime} if it has no nontrivial homogeneous sets.
We call a subset $S$ of the vertex set of a graph a \textit{stable set} if all vertices in $S$ are pairwise nonadjacent.
A \textit{clique} is a complete subgraph.
A graph is called a \textit{split graph} if its vertex set can be partitioned into a clique and a stable set.

We now list some results on vertex-critical graphs that will be helpful throughout the remaining sections.

     \begin{lemma}[\cite{KCameron2021}] \label{lem:XY}
		Let $G$ be a $k$-vertex-critical graph. There does not exist two vertex subsets $X, Y \subseteq V(G)$ satisfying all of the following conditions.
		\begin{itemize}
			\item $X$ and $Y$ are anticomplete to each other.
			\item $\chi(G[X])\le\chi(G[Y])$.
			\item Y is complete to $N(X)$.
		\end{itemize}
	\end{lemma}


\begin{lemma}[\cite{Hoang2015}]\label{lem:homosetscritical}
    If $G$ is a $k$-vertex-critical graph and $S\subset V(G)$ is a homogeneous set, then $S$ induces an $m$-vertex-critical graph for some $m<k$.
\end{lemma}

\section{Unavoidable induced subgraphs in large $k$-vertex critical graphs}\label{sec:unavoidablestructures}



We begin this section by proving a result on prime graphs and vertex-critical graphs.
We note that our result follows directly from Theorem 3.2 of~\cite{BelavadiHoang2026}, but we include its full proof here for completeness. 

\begin{lemma}\label{lem:finitekcolorableprimeimpliesfinitekcritical}
    Let $\mathcal{G}$ be a hereditary class of graphs and $k$ be a fixed non-negative integer. 
    If the order of every $k$-colourable prime graph in $\mathcal{G}$ is bounded by a constant, 
    then there are only finitely many $k$-vertex-critical graphs in $\mathcal{G}$.
\end{lemma}
\begin{proof}
    The proof is by induction on $k$.
    For $k=1$, the result is trivial for all graph classes.
    Now suppose for all $1\le k'< k$ that if the order of every $k'$-colourable prime graph in $\mathcal{G}$ is bounded by a constant, 
    then there are only finitely many $k'$-vertex-critical graphs in $\mathcal{G}$.
    Let $N_{k'}$ be the maximum order of a $k'$-colourable prime graph in $\mathcal{G}$ and $M_{k'}$ be the maximum order of a $k'$-vertex-critical graph in $\mathcal{G}$.

    Now, suppose that $N_k$ is bounded by a constant.
    Since every $k$-colourable graphs is also $k'$-colourable, and $\mathcal{G}$ is a hereditary class, this implies that $N_{k'}$ is bounded by a constant for all $k'<k$.
    Thus, by the inductive hypothesis, there are only finitely many $k'$-vertex-critical graphs in $\mathcal{G}$ for all $1\le k'<k$, and therefore, $M_{k'}$ is bounded by a constant.
    Further, since $N_k$ is bounded by a constant, the order every prime $k$-vertex-critical graph in $\mathcal{G}$ is bounded by a constant. 
    Thus, it remains only to bound the order of every non-prime $k$-vertex-critical graph in $\mathcal{G}$.
    Let $G$ be a $k$-vertex-critical graph in $\mathcal{G}$ that is not prime, and let $H$ be the prime graph in $\mathcal{G}$ that is obtained from $G$ by deleting all but one vertex in every nontrivial homogeneous set in $G$.
    Since $G$ is $k$-colourable, it follows that $H$ is $k$-colourable and therefore has order at most $N_k$.
    From Lemma~\ref{lem:homosetscritical}, each nontrivial homogeneous set in $G$ must be $k'$-vertex-critical for some $k'<k$. 
    Thus, $|V(G)|\le N_{k}\cdot \max_{1\le k' < k}(M_{k'})$, and the order of every non-prime $k$-vertex-critical graph in $\mathcal{G}$ is bounded by a constant.
    Therefore, there are only finitely many $k$-vertex-critical graphs in $\mathcal{G}$.
     The result follows by induction.
    
\end{proof}

In~\cite{CHUDNOVSKY20161}, Chudnovsky, Kim, Oum, and Seymour give a list of unavoidable graphs in large prime graphs. Before stating their powerful theorem that we will use throughout, we need to give several definitions. See Figure~\ref{fig:unlabelledprimegraphs} for examples of some of the following graphs and their respective complements.

\begin{itemize}
    \item The \emph{thin spider} with $n$ legs is a split graph on $2n$ vertices consisting of a stable set $\{a_1, a_2,\ldots, a_n\}$ and a clique $\{b_1, b_2, \ldots, b_n\}$ such that $a_i \sim b_j$ if and only if $i =j$.
    \item The \emph{half-graph} of height $n$, denoted $H_n$, is a bipartite graph on $2n$ vertices $a_1, a_2, \ldots, a_n, b_1, b_2, \ldots, b_n$ such that  $a_i \sim b_j$ if and only if $i \geq j$..
    \item The graph $H_{n,I}'$ is obtained from $H_n$ by making $\{b_1, b_2, \ldots, b_n\}$ a clique and adding a new vertex adjacent to $a_1, a_2, \ldots, a_n$.
    \item The graph $H_{n}^{*}$ is obtained from $H_n$ by making $\{b_1, b_2, \ldots, b_n\}$ a clique and adding a new vertex adjacent to only $a_n$.
    \item A sequence of distinct vertices $v_0, v_1, \ldots, v_n$ is a \emph{chain} of length $n$ if for $i\geq 1$ we have that $v_{i-1}$ is the unique neighbour or the unique non-neigbhour of $v_i$ among the vertices preceding $v_i$ in the chain (i.e. the vertices $\{v_0, v_1, \dots, v_{i-1}\}$).
\end{itemize}

\begin{theorem}[\cite{CHUDNOVSKY20161}]\label{thm:unavoidableprimegraphs}
    For every integer $n\ge 3$, there exists a positive integer $N$ such that every prime graph with at least $N$ vertices contains one of the following graphs or their complements as an induced subgraph.
    \begin{enumerate}[label=\alph*)]
        \item The $1$-subdivision of $K_{1,n}$, denoted $K_{1,n}^{(1)}$.
        \item The line graph of $K_{2,n}$, denoted $L(K_{2,n})$.
        \item The thin spider with $n$ legs.
        \item The half-graph of height $n$, denoted $H_n$.
        \item The graph $H_{n,I}'$.
        \item The graph $H_{n}^{*}$.
        \item A prime graph induced by a chain of length $n$.
    \end{enumerate}
\end{theorem}

\begin{figure}[!h]
\def\c{0.65}

\centering
\subfigure[$K_{1,5}^{(1)}$]{
\scalebox{\c}{
\begin{tikzpicture}[
    vertex/.style={circle,fill=black,inner sep=1.6pt},
    edge/.style={thin}
]
\begin{scope}[shift={(0,0)}]
\node[vertex] (x) at (-0.4,0.75) {};

\foreach \i in {1,2,3,4,5}{
    \pgfmathsetmacro{\y}{3-\i*0.75}
    \node[vertex] (a\i) at (0.8,\y) {};
    \node[vertex] (b\i) at (2.2,\y) {};
    \draw[edge] (x)--(a\i)--(b\i);
}

\end{scope}
\end{tikzpicture}}
}
\qquad
\subfigure[$\overline{K_{1,5}^{(1)}}$]{
\scalebox{\c}{
\begin{tikzpicture}[
    vertex/.style={circle,fill=black,inner sep=1.6pt},
    edge/.style={thin}
]
\begin{scope}[shift={(0,0)}]
\node[vertex] (x) at (-0.4,0.75) {};

\foreach \i in {1,...,5}{
    \pgfmathsetmacro{\y}{3-\i*0.75}
    \node[vertex] (v\i) at (0.8,\y) {};
    \node[vertex] (u\i) at (2.2,\y) {};
    \draw[edge] (x)--(v\i);
}

\foreach \i in {1,...,4}{
    \pgfmathsetmacro{\ip}{\i+1}
    \foreach \j in {\ip,...,5}{
        \draw (u\i)--(v\j);
    }
}

\foreach \i in {2,...,5}{
    \pgfmathsetmacro{\im}{\i-1}
    \foreach \j in {1,...,\im}{
        \draw (u\i)--(v\j);
    }
}

\foreach \i in {1,...,5}{
    \pgfmathsetmacro{\ip}{\i+1}
    \foreach \j in {\i,...,5}{
        \draw (u\i) to [bend left=40] (u\j);
        \draw (v\i) to [bend right=40] (v\j);
    }
}

\end{scope}
\end{tikzpicture}}
}
\qquad
\subfigure[$L(K_{2,5})$]{
\scalebox{\c}{
\begin{tikzpicture}[
    vertex/.style={circle,fill=black,inner sep=1.6pt},
    edge/.style={thin}
]

\foreach \i in {1,...,5}{
    \pgfmathsetmacro{\y}{3-\i*0.75}
    \node[vertex] (a\i) at (0.8,\y) {};
    \node[vertex] (b\i) at (2.2,\y) {};
    \draw[edge] (a\i)--(b\i);
}

\foreach \i in {1,...,5}{
    \foreach \j in {\i,...,5}{
        \draw (b\i) to [bend left=40] (b\j);
        \draw (a\i) to [bend right=40] (a\j);
    }
}
\end{tikzpicture}}
}
\qquad
\subfigure[$\overline{L(K_{2,5})}$]{
\scalebox{\c}{
\begin{tikzpicture}[
    vertex/.style={circle,fill=black,inner sep=1.6pt},
    edge/.style={thin}
]

\foreach \i in {1,...,5}{
    \pgfmathsetmacro{\y}{3-\i*0.75}
    \node[vertex] (a\i) at (0.8,\y) {};
    \node[vertex] (b\i) at (2.2,\y) {};
}

\foreach \i in {1,...,4}{
    \pgfmathsetmacro{\ip}{\i+1}
    \foreach \j in {\ip,...,5}{
        \draw (a\i)--(b\j);
    }
}

\foreach \i in {2,...,5}{
    \pgfmathsetmacro{\im}{\i-1}
    \foreach \j in {1,...,\im}{
        \draw (a\i)--(b\j);
    }
}
\end{tikzpicture}}
}
\qquad
\subfigure[$K_n^*$]{
\scalebox{\c}{
\begin{tikzpicture}[
    vertex/.style={circle,fill=black,inner sep=1.6pt},
    edge/.style={thin}
]

\foreach \i in {1,...,5}{
    \pgfmathsetmacro{\y}{3-\i*0.75}
    \node[vertex] (a\i) at (0.8,\y) {};
    \node[vertex] (b\i) at (2.2,\y) {};
    \draw[edge] (a\i)--(b\i);
}

\foreach \i in {1,...,5}{
    \foreach \j in {\i,...,5}{
        \draw (b\i) to [bend left=40] (b\j);
    }
}
\end{tikzpicture}}
}
\qquad
\subfigure[$\overline{K_n^*}$]{
\scalebox{\c}{
\begin{tikzpicture}[
    vertex/.style={circle,fill=black,inner sep=1.6pt},
    edge/.style={thin}
]

\foreach \i in {1,...,5}{
    \pgfmathsetmacro{\y}{3-\i*0.75}
    \node[vertex] (a\i) at (0.8,\y) {};
    \node[vertex] (b\i) at (2.2,\y) {};
}

\foreach \i in {1,...,5}{
    \foreach \j in {\i,...,5}{
        \draw (b\i) to [bend left=40] (b\j);
    }
}

\foreach \i in {1,...,4}{
    \pgfmathsetmacro{\ip}{\i+1}
    \foreach \j in {\ip,...,5}{
        \draw (a\i)--(b\j);
    }
}

\foreach \i in {2,...,5}{
    \pgfmathsetmacro{\im}{\i-1}
    \foreach \j in {1,...,\im}{
        \draw (a\i)--(b\j);
    }
}
\end{tikzpicture}}
}
\qquad
\subfigure[$H_5$]{
\scalebox{\c}{
\begin{tikzpicture}[
    vertex/.style={circle,fill=black,inner sep=1.6pt},
    edge/.style={thin}
]

\foreach \i in {1,...,5}{
    \pgfmathsetmacro{\y}{3-\i*0.75}
    \node[vertex] (a\i) at (0.8,\y) {};
    \node[vertex] (b\i) at (2.2,\y) {};
}

\foreach \i in {1,...,5}{
    \foreach \j in {\i,...,5}{
        \draw (b\i)--(a\j);
    }
}
\end{tikzpicture}}
}
\qquad
\subfigure[$\overline{H_5}$]{
\scalebox{\c}{
\begin{tikzpicture}[
    vertex/.style={circle,fill=black,inner sep=1.6pt},
    edge/.style={thin}
]

\foreach \i in {1,...,5}{
    \pgfmathsetmacro{\y}{3-\i*0.75}
    \node[vertex] (a\i) at (0.8,\y) {};
    \node[vertex] (b\i) at (2.2,\y) {};
}

\foreach \i in {1,...,5}{
    \foreach \j in {\i,...,5}{
        \draw (b\i) to [bend left=40] (b\j);
        \draw (a\i) to [bend right=40] (a\j);
    }
}

\foreach \i in {1,...,4}{
    \pgfmathsetmacro{\ip}{\i+1}
    \foreach \j in {\ip,...,5}{
        \draw (a\i)--(b\j);
    }
}
\end{tikzpicture}}
}
\qquad
\subfigure[$H_{n,I}'$]{
\scalebox{\c}{
\begin{tikzpicture}[
    vertex/.style={circle,fill=black,inner sep=1.6pt},
    edge/.style={thin}
]

\node[vertex] (x) at (-0.4,0.75) {};

\foreach \i in {1,2,3,4,5}{
    \pgfmathsetmacro{\y}{3-\i*0.75}
    \node[vertex] (a\i) at (0.8,\y) {};
    \node[vertex] (b\i) at (2.2,\y) {};
    \draw[edge] (x)--(a\i);
}

\foreach \i in {1,...,5}{
    \foreach \j in {\i,...,5}{
        \draw (b\i) to [bend left=40] (b\j);
    }
}

\foreach \i in {1,...,5}{
    \foreach \j in {\i,...,5}{
        \draw (b\i)--(a\j);
    }
}
\end{tikzpicture}}
}
\qquad
\subfigure[$\overline{H_{n,I}'}$]{
\scalebox{\c}{
\begin{tikzpicture}[
    vertex/.style={circle,fill=black,inner sep=1.6pt},
    edge/.style={thin}
]
\node[vertex] (x) at (3.6,0.75) {};

\foreach \i in {1,2,3,4,5}{
    \pgfmathsetmacro{\y}{3-\i*0.75}
    \node[vertex] (a\i) at (0.8,\y) {};
    \node[vertex] (b\i) at (2.2,\y) {};
    \draw[edge] (x)--(b\i);
}

\foreach \i in {1,...,5}{
    \foreach \j in {\i,...,5}{
        \draw (a\i) to [bend right=40] (a\j);
    }
}

\foreach \i in {1,...,4}{
    \pgfmathsetmacro{\ip}{\i+1}
    \foreach \j in {\ip,...,5}{
        \draw (a\i)--(b\j);
    }
}
\end{tikzpicture}}
}
\qquad
\subfigure[$H_5^*$]{
\scalebox{\c}{
\begin{tikzpicture}[
    vertex/.style={circle,fill=black,inner sep=1.6pt},
    edge/.style={thin}
]

\node[vertex] (x) at (-0.4,1.1) {};

\foreach \i in {1,2,3,4,5}{
    \pgfmathsetmacro{\y}{3-\i*0.75}
    \node[vertex] (a\i) at (0.8,\y) {};
    \node[vertex] (b\i) at (2.2,\y) {};

}
\draw[edge] (x)--(a5);
\foreach \i in {1,...,5}{
    \foreach \j in {\i,...,5}{
        \draw (b\i) to [bend left=40] (b\j);
    }
}

\foreach \i in {1,...,5}{
    \foreach \j in {\i,...,5}{
        \draw (b\i)--(a\j);
    }
}
\end{tikzpicture}}
}
\qquad
\subfigure[$\overline{H_5^*}$]{
\scalebox{\c}{
\begin{tikzpicture}[
    vertex/.style={circle,fill=black,inner sep=1.6pt},
    edge/.style={thin}
]
\node[vertex] (x) at (-0.4,1.1) {};

\foreach \i in {1,2,3,4,5}{
    \pgfmathsetmacro{\y}{3-\i*0.75}
    \node[vertex] (a\i) at (0.8,\y) {};
    \node[vertex] (b\i) at (2.2,\y) {};
    \draw[edge] (x)--(b\i);
}

\foreach \i in {1,...,4}{
    \draw[edge] (x)--(a\i);
}

\foreach \i in {1,...,5}{
    \foreach \j in {\i,...,5}{
        \draw (a\i) to [bend right=40] (a\j);
    }
}

\foreach \i in {1,...,4}{
    \pgfmathsetmacro{\ip}{\i+1}
    \foreach \j in {\ip,...,5}{
        \draw (a\i)--(b\j);
    }
}
\end{tikzpicture}}
}
\caption{Examples of graphs and their complements listed in Theorem \ref{thm:unavoidableprimegraphs}}\label{fig:unlabelledprimegraphs}
\end{figure}
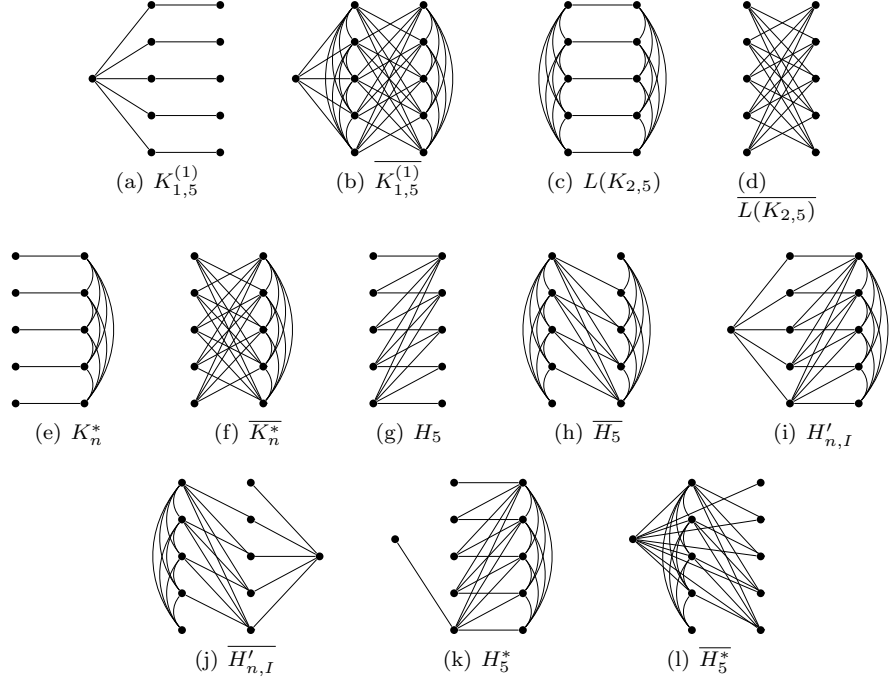

Note that several graphs in Theorem \ref{thm:unavoidableprimegraphs} contain $K_n$ as an induced subgraph. One notable exception is a chain. However this next lemma shows long chains contain either a path or clique.

\begin{lemma}\label{lem:chainhaspathorclique}
    Every chain of length $m\ge 2nk$ contains an induced $P_n$ or an induced $K_k$.
\end{lemma}
\begin{proof}
    Let $G$ be a graph induced by a chain of length $m\ge 2nk$ with vertices labeled $v_0$ to $v_m$ as in the definition of a chain above.
    Suppose for some $i\in\{1,2,\dots , m-n \}$, we have, for each $j\in \{i,i+1,i+2,\dots, i+n\}$, that $v_{j-1}$ is the unique neighbour of $v_{j}$ in $\{v_0,v_1,\dots v_{j-1}\}$. 
    Then  $\{v_{i-1},v_{i},v_{i+1},\dots, v_{i+n-1}\}$ induces a $P_n$ in $G$.
    So we may suppose that this never occurs.
    
    For each $i\in\{0,1,2,\dots 2k-1\}$, define the subsets $V_i=\{v_{ni+1},v_{ni+2}\dots v_{n(i+1)}\}$ of $V(G)$.
    Thus, for each $i\in\{0,1,2,\dots 2k-1\}$, we must have a vertex $v_{i'}\in V_i$ such that $v_{i'-1}$ is the unique non-neighbour of $v_{i'}$ in the set $\{v_0,v_1,\dots v_{i'-2},v_{i'-1}\}$.
    Thus, for all $i\in\{0,1,2,\dots 2k-1\}$, $v_{i'}\sim v_{\ell'}$ for all $\ell < i-1$.
    Further, $v_{i'}\nsim v_{i'-1}$ if and only if $i'=ni+1$ and $i'-1 = ni$.
    However, we can have $i'=ni+1$ and $i'-1 = ni$ for at most $k$ values of $i$ since we then necessarily have $i'-1\neq n(i-1)+1$ and $i'\neq n(i+1)$.
    Thus, there is a subset $I\subseteq \{0,1,2\dots 2k-1\}$ such that $|I|\ge k$ and for all $i_1,i_2\in I$ we have $v_{i_1'}\sim v_{i_2'}$. 
    Thus, $G$ contains an induced $K_k$.
\end{proof}

Lemma \ref{lem:chainhaspathorclique} and the fact that all but $K_{1,2nk}^{(1)}$, $\overline{L(K_{2,2nk})}$, and $H_{2nk}$ contain an induced $K_k$ give the following immediate corollary of Theorem \ref{thm:unavoidableprimegraphs}.

\begin{corollary}\label{cor:unavoidableprimegraphs}
    For any integers $k$ and $n\geq 3$, there are only finitely many prime graph which do not contain any of $K_{1,2nk}^{(1)}$, $\overline{L(K_{2,2nk})}$, $H_{2nk}$, $P_n$, nor $K_k$ as an induced subgraph.
\end{corollary}

\begin{table}[h!]
    \centering
    \begin{tabular}{|c|c|c|}
    \hline
    Graph & Induced Subgraph & Vertex Set \\ \hline
                    & $P_5$ & $\{b_1,a_1, x, a_2, b_2 \}$ \\
    $K_{1,n}^{(1)}$ & $P_5+(n-2)P_1$ & $\{b_1,a_1, x, a_2, b_2, b_3, \ldots, b_n \}$ \\
                    & $P_4+(n-2)P_1$ & $\{b_1,a_1, x, a_2, b_3, \ldots, b_n \}$\\ \hline
                            & $P_5$ & $\{a_2,b_1, a_3, b_2, a_1 \}$ \\
    $\overline{L(K_{2,n})}$ & $C_6$ & $\{a_2,b_1, a_3, b_2, a_1, b_3 \}$ \\ \hline
          & $P_4+(n-2)P_1$ & $\{a_1,b_1, a_2, b_2, b_3, \ldots, b_n \}$\\
    $H_n$ & $C_4+(n-3)P_1$ & $\{b_1, a_2, b_2,a_3, b_4, \ldots, b_n \}$\\ 
          \hline
   
    \end{tabular}
    \caption{Induced subgraphs in $K_{1,n}^{(1)}$, $\overline{L(K_{2,n})}$, and $H_n$}
    \label{tab:subgraphs}
\end{table}



\setcounter{subfigure}{0}
\begin{figure}[!h]
\def\c{1}

\centering
\subfigure[$K_{1,5}^{(1)}$]{
\scalebox{\c}{
\begin{tikzpicture}[
    vertex/.style={circle,fill=black,inner sep=1.6pt},
    edge/.style={thin}
]
\begin{scope}[shift={(0,0)}]
\node[vertex, label=left:$x$] (x) at (-0.4,0.75) {};

\foreach \i in {1,2,3,4}{
    \pgfmathsetmacro{\y}{3-\i*0.75}
    \node[vertex, label=below:$a_{\i}$] (a\i) at (0.8,\y) {};
    \node[vertex, label=right:$b_{\i}$] (b\i) at (2.2,\y) {};
    \draw[edge] (x)--(a\i)--(b\i);
}

\foreach \i in {5}{
    \pgfmathsetmacro{\y}{3-\i*0.75}
    \node[vertex, label=left:$a_{\i}$] (a\i) at (0.8,\y) {};
    \node[vertex, label=right:$b_{\i}$] (b\i) at (2.2,\y) {};
    \draw[edge] (x)--(a\i)--(b\i);
}

\end{scope}
\end{tikzpicture}}
}
\qquad
\subfigure[$\overline{L(K_{2,5})}$]{
\scalebox{\c}{
\begin{tikzpicture}[
    vertex/.style={circle,fill=black,inner sep=1.6pt},
    edge/.style={thin}
]

\foreach \i in {1,...,5}{
    \pgfmathsetmacro{\y}{3-\i*0.75}
    \node[vertex, label=left:$a_{\i}$] (a\i) at (0.8,\y) {};
    \node[vertex, label=right:$b_{\i}$] (b\i) at (2.2,\y) {};
}

\foreach \i in {1,...,4}{
    \pgfmathsetmacro{\ip}{\i+1}
    \foreach \j in {\ip,...,5}{
        \draw (a\i)--(b\j);
    }
}

\foreach \i in {2,...,5}{
    \pgfmathsetmacro{\im}{\i-1}
    \foreach \j in {1,...,\im}{
        \draw (a\i)--(b\j);
    }
}
\end{tikzpicture}}
}
\qquad
\subfigure[$H_5$]{\label{subfig:halfgraph5labelled}
\scalebox{\c}{
\begin{tikzpicture}[
    vertex/.style={circle,fill=black,inner sep=1.6pt},
    edge/.style={thin}
]

\foreach \i in {1,...,5}{
    \pgfmathsetmacro{\y}{3-\i*0.75}
    \node[vertex, label=left:$a_{\i}$] (a\i) at (0.8,\y) {};
    \node[vertex, label=right:$b_{\i}$] (b\i) at (2.2,\y) {};
}

\foreach \i in {1,...,5}{
    \foreach \j in {\i,...,5}{
        \draw (b\i)--(a\j);
    }
}
\end{tikzpicture}}
}
\caption{Three labeled graphs from Corollary \ref{cor:unavoidableprimegraphs}}\label{fig:labelledprime}
\end{figure}
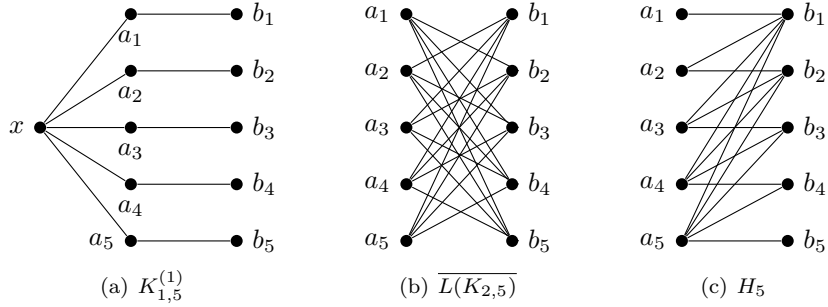

Consider the labeled examples of $K_{1,n}^{(1)}$, $\overline{L(K_{2,n})}$, and $H_n$ for $n=5$ in Figure \ref{fig:labelledprime}. Table \ref{tab:subgraphs} contains a list of common induced subgraphs in each graph together with a vertex set which induces the subgraph. 
The following theorems now follow almost directly

\begin{theorem}\label{thm:P5Hn}
    For all $k,n\ge 1$, $\numcrit{k}{P_5,H_{n}}<\infty$.
\end{theorem}
\begin{proof}
Consider the collection of all $(P_5, H_n)$-free $k$-critical graphs for some fixed $k,n \geq 1$.
From Lemma~\ref{lem:finitekcolorableprimeimpliesfinitekcritical}, it suffices to show there are only finitely many $k$-colourable $(P_5, H_n)$-free prime graphs.
From Corollary \ref{cor:unavoidableprimegraphs}, there are only finitely many prime graphs that forbid $K_{1,2mr}^{(1)}$, $\overline{L(K_{2,2mr})}$, $H_{2mr}$, $P_{m}$, and $K_r$ as an induced subgraph.
So, choose $m \geq 5$ and $r\geq k+1$ such that $2mr \geq n$.
From Table \ref{tab:subgraphs} we have that $K_{1,2mk}^{(1)}$, $\overline{L(K_{2,2mk})}$, and $P_{m}$ each contain an induced $P_5$.
Moreover $H_{2mk}$ contains an induced $H_n$ and no $k$-colourable graph contains an induced $K_r$.
Therefore there are only finitely many $k$-colourable $(P_5, H_n)$-free prime graphs.   
\end{proof}

The following two corollaries follow immediately from Theorem~\ref{thm:P5Hn} and Table~\ref{tab:subgraphs}. 

\begin{corollary}\label{cor:P5P4+ellP1}
    For all $k,\ell\ge 1$, $\numcrit{k}{P_5,P_4+\ell P_1}<\infty$.
\end{corollary}

\begin{corollary}\label{cor:P5C4+ellP1}
    For all $k,\ell\ge 1$, $\numcrit{k}{P_5,C_4+\ell P_1}<\infty$.
\end{corollary}

The above two corollaries are notable as for $\ell=1$ they resolve two of the five remaining open cases of Problem~\ref{prob:P5Hord5}, but they are far from the only corollaries. 
We will now show that $\numcrit{k}{P_5,H}<\infty$ for every $2P_2$-free bipartite graph $H$.
Recall that in \cite{Hoang2015} it was shown that $\numcrit{k}{2P_2}=\infty$-free graphs for all $k\ge 5$. 
Additionally, it is known that the half-graph $H_n$ is a universal graph for $2P_2$-free bipartite graphs~\cite{LozinRudol2027}. 
That is every $2P_2$-free bipartite graph is an induced subgraph of $H_n$ for some $n$. 
Thus, the following theorem follows immediately. 
We will provide a short proof for completeness.

\begin{theorem}\label{thm:P5Bipartite}
    Let $H$ be a bipartite graph and $k \geq 5$. Then $\numcrit{k}{P_5,H}<\infty$ if and only if $H$ is $2P_2$-free.
\end{theorem}
\begin{proof}
From \cite{Hoang2015}, it suffices to show there are only finitely many $k$-vertex-critical $(P_5, H)$-free graphs where $H$ is bipartite and $2P_2$-free.
From Theorem \ref{thm:P5Hn}, there are finitely many $k$-vertex-critical $(P_5, H_n)$-free graphs for any fixed $n$.
So, it suffices to show for some large enough $n$ that $H$ is an induced subgraph of $H_n$.
Let $A=\{a_1, \ldots,a_m \}$ and $B=\{b_1, \ldots,b_r \}$ be a bipartition of $H$ where all isolated vertices are in $B$.
Without loss of generality relabel the vertices of $A$ in increasing order of degree such that $a_1$ has the minimum degree in $A$.
Since $H$ is $2P_2$-free then for any two $a_i, a_j \in A$ either $N(a_i) \subseteq N(a_j)$ or $N(a_j) \subseteq N(a_i)$.
Therefore 

$$N(a_1) \subseteq N(a_2) \subseteq \cdots \subseteq N(a_m),$$

\noindent with $N(a_1) = N(a_2)$ if and only if  $|N(a_1)| = |N(a_2)|$.

Now partition $A = A_1 \cup A_2 \cup \cdots \cup A_d$ by the degree of each vertex such that $A_1$ is all vertices of lowest degree in $A$ and $A_d$ is all vertices of highest degree in $A$.
Now let $B_1=N(A_1)$, $B_2=N(A_2)/N(A_1)$, $B_2=N(A_3)/N(A_2)$, and so on.
Additionally let $B_0$ be all vertices in $H$ which are isolated.
Note that for each $i\geq 1$, the vertices in $B_i$ are joined to every vertex in each $A_j$ with $j\geq i$.
Now consider the half-graph $H_n$ with $n = m+r$ as labelled in Figure \ref{subfig:halfgraph5labelled}.
To avoid confusion we will use $a_i'$ and $b_i'$ when referring to vertices in $H_n$. We will now choose a set of vertices $S$ which induce $H$ in $H_n$.
For each of the $n$ pairs $\{a_i', b_i'\}$ in $H_n$ we will choose exactly one of $a_i$ or $b_i$ to be in $S$.
Working from the lowest index to the highest index in $H_n$ alternate choosing $|B_1|$ vertices form the $B'$ side to be in $S$, then the next $|A_1|$ vertices form the $A'$ side, then the next $|B_2|$ vertices form the $B'$ side and so on. At each iteration, label the set of $|A_i|$ or $|B_i|$ chosen vertices to be $A_i'$ or $B_i'$.
This should continue until we have chosen $|A_d|$ vertices form the $A'$ side to be in $S$.
At this point there are exactly $|B_0|$ vertices remaining which we choose from the $B'$ side to be in $S$.
Now consider the subgraph of $H_n$ induced by $S$.
Now for each $i\geq 1$, the vertices in $B_i'$ are joined to every vertex in each $A_j'$ with $j\geq i$.
Additionally the vertices in $B_0$ are isolated.
Therefore $H$ is an induced graph of $H_n$.

\end{proof}

\begin{theorem}\label{thm:P4ellP1Cocktail}
    For all $\ell \ge 0$ and $k,n\ge 1$, $\numcrit{k}{P_4+\ell P_1,\overline{L(K_{2,n})}}<\infty$.
\end{theorem}
\begin{proof}
Consider the collection of all $(P_4+\ell P_1, \overline{L(K_{2,n})})$-free $k$-critical graphs for some fixed $k,n \geq 1$.
From Lemma~\ref{lem:finitekcolorableprimeimpliesfinitekcritical}, it suffices to show there are only finitely many $k$-colourable $(P_4+\ell P_1, \overline{L(K_{2,n})})$-free prime graphs.
From Corollary \ref{cor:unavoidableprimegraphs}, there are only finitely many prime graphs which forbid $K_{1,2mr}^{(1)}$, $\overline{L(K_{2,2mr})}$, $H_{2mr}$, $P_{m}$, and $K_r$ as an induced subgraph.
Choose $m \geq 2\ell+4$ and $r\geq k+1$ such that $2mr \geq n$.
From Table \ref{tab:subgraphs}, we have that $K_{1,2mk}^{(1)}$, $H_{2mk}$, and $P_{m}$ each contain an induced $P_4+\ell P_1$.
Moreover, $\overline{L(K_{2,2mk})}$ contains an induced $\overline{L(K_{2,n})}$ and no $k$-colourable graph contains an induced $K_r$.
Therefore there are only finitely many $k$-colourable $(P_4+\ell P_1, \overline{L(K_{2,n})})$-free prime graphs.   
\end{proof}

\section{($P_5$, $H$)-free graphs for non-bipartite $H$}\label{sec:P5HnonbipartiteH}

In this section, we make some progress on Problem~\ref{prob:P5Hlargerchrom}.
We begin by showing one large class of graphs $H$ where we have finiteness.
The rest of the section is devoted to finding new structure in previously known infinite families to prove that there are infinitely many $k$-vertex-critical $(P_5,H)$-free graphs for new graphs $H$.

\begin{theorem}\label{thm:P5co-(Kell+P3)-free}
    For all $k, s\geq 1$ and $t \geq 2$, $\numcrit{k}{P_5, K_{s,t}+e}<\infty$ where $K_{s,t}+e$ is a complete bipartite graph with an edge added between two of the $t$ vertices in one set of the bipartition.
\end{theorem}

\begin{proof}
    
    Let $G$ be a $k$-vertex-critical $(P_5, K_{s,t}+e)$-free graph.
    If $G$ is $H_n$-free for any fixed $n \geq 1$, then we are done from Theorem~\ref{thm:P5Hn}.
    So suppose $G$ contains an induced $H_n$ for some $n \geq s+(t-4)R(k,t)+1$ with the same labeling established in Figure~\ref{subfig:halfgraph5labelled}.
    Let $v_i \not\in V(H_n)$ be some vertex such that for two vertices $b_i, b_j \in V(H_n)$ with $i<j \leq n-s-1$ we have $v_i \not\sim b_i$ and $v_i \sim b_{j}$.
    Note that for each pair of vertices $b_i$ and $b_j$ with $i<j$, such a vertex $v_i$ exists from Lemma~\ref{lem:XY} as otherwise we would have $N(b_j) \subseteq N(b_i)$ for any $j>i$.
    We will now show a general claim for any given $v_i$.
    We claim that $v_i$ is adjacent to all but at most $t-3$ vertices in $\{b_1, \ldots, b_{n-s-1}\}$.

    First suppose for any two indices $j\leq x < y$ that $v_i \not\sim a_x$ and $v_i \not\sim a_y$.
    We will show that this implies that $G$ contains an induced $P_5$.
    Thus $v_i \sim a_i$ otherwise $\{a_i, b_i, a_x, b_{j}, v_i\}$ would induce a $P_5$.
    As $x <y$ then $a_x \not\sim b_y$.
    Moreover, $v_i \not\sim b_{y}$ otherwise $\{b_i, a_x, b_{j}, v_i, b_y\}$ would induce a $P_5$.
    However now $\{a_i, v_i, b_j, a_y, b_y\}$ induces a $P_5$.
    Therefore each $v_i$ is adjacent to all but at most one vertex in $\{a_j, a_{j+1}, \ldots, a_n\}$.
    Hence $v_i$ is adjacent to at least $s$ vertices in $\{a_{n-s},a_{n-s+1}, \dots, a_{n}\}$.
    Let $A_s$ denote the set of those $s$ neighbours.
    Now suppose $v_i$ is not adjacent to $t-2$ vertices in $\{b_1, \ldots, b_{n-s-1}\}$.
    Let $B_{t-2}$ denote the set of those $t-2$ neighbours.
    Note that $B_{t-2} \cup \{v_i, b_j\}$ induces a graph with $t$ vertices and exactly one edge $\{v_i, b_j\}$.
    Additionally $B_{t-2} \cup \{v_i, b_j\}$ is joined to $A_s$.
    Thus $A_s \cup B_{t-2} \cup \{v_i, b_j\}$ induces a $K_{s,t}+e$.
    Therefore $v_i$ is adjacent to all but at most $t-3$ vertices in $\{b_1, \ldots, b_{n-s-1}\}$.
    Note when $t=2$ or $t=3$ this is a contradiction as $v_i \not\sim b_i$ and hence is adjacent to all but at least one vertex in $\{b_1, \ldots, b_{n-s-1}\}$.

    Now recall that for each $i\in\{1,2,\dots,n-s-1\}$, there exists a vertex $v_i$ for each $b_i$ such that $v_i \not\sim b_i $ but $v_i \sim b_j$.
    Note each $v_i$ are not necessarily distinct.
    We will form the sets $V'$ and $B'$ as follows.
    Initialize $V_1'=\{v_{i_1}\}$ for $i_1=1$ and $B_1'=\{b_1, \ldots, b_{n-s-1}\}$.
    For each $m\ge 2$, let 
    $$B_{m}'=B_{m-1}'\setminus \{b_j\in B_{m-1}': v_{i_{m-1}}\nsim b_j \text{ and } j > i_{m-1}\}.$$
    
    Next, choose the smallest index $i_{m}$ such that $v_{i_{m-1}}\sim b_{i_m}$ and let $V_m'=V_{m-1}'\cup \{v_{i_m}\}$.
    Note that $b_{i_m} \in B_m'$.
    Repeat this process iteratively until $m=R(k,t)$.
    Note, we can do this since $n-s-1\ge (t-4)R(k,t)$ and each $v_{i_j}$ has at most $t-3$ non-neighbours in $B_{j-1}'$, including $b_{i_j}$, which remains in the set $B_{j}'$.
    As $m = R(k,t)$, we may also assume that there is subset $V'' \subseteq V_m'$ with $|V''| =t$ which is a stable set otherwise $\chi(G) \geq k$.
    So take $V''$ and $B''=\{b_i\in B_m': v_i\in V''\}$.
    Note that for each $v_i \in V''$, by the construction of $B''$, we have that $v_i \not\sim b_i$ and $v_i \sim b_j$ for all $b_j \in B''$ with $j\geq i$.
    Now relabel the $V''$ and $B''$ as $\{v_1, \ldots , v_{t}\}$ and $\{b_1, \ldots , b_{t}\}$ which preserves the order of the original indices.
    Note that $v_{t-1}$ still has at most $t-3$ non-neighbours in $B''$.
    Thus for some $x \in \{1, \ldots, t-2 \}$ we have $v_{t-1} \sim b_x$.
    Moreover we have $v_x \sim b_{t-1}$, $v_x \sim b_{t}$, and $v_{t-1} \sim b_{t}$.
    Now, $\{b_x, v_{t-1}, b_t, v_x, b_{t-1}\}$ induces a $P_5$, a contradiction.
    
\end{proof}


    

We now move to finding new structure in a previous infinite family of $k$-vertex-critical $(P_5, C_5)$-free construction.

\begin{definition}[\cite{CameronHoang2023P5C5}]\label{def:G(q,k)}
Fix $q\ge 1$ and $k\ge 3$. Let $\infcrit{q}{k}$ be a graph on vertex set $\{v_0,v_1,...,v_{kq}\}$ where, with each integer taken modulo $kq+1$, the neighbourhood of vertex $v_i$ is $$\{v_{i-1},v_{i+1}\}\cup\{v_{i+kj+m} : m=2,3,...k-1\text{ and } j=0,1,...,q-1\}.$$ 
\end{definition}

See Figure~\ref{fig:Gqkneighbandfullgraph} for a depiction of the neighbourhood of $v_0$ in $\infcrit{6}{6}$ and for the full graph $\infcrit{3}{7}$. Throughout the proofs that follow it is helpful to note that it follows directly from Definition~\ref{def:G(q,k)} that the non-neighbours of every vertex $v_{i}$ in $V(\infcrit{q}{k})$ are every vertex with label $i$ or $i+1\pmod{k}$ in the set $\{v_{i+k},v_{i+k+1},v_{i+k+2},\dots ,v_{qk}\}$ and every vertex with label $i-1$ or $i\pmod{k}$ in the set $\{v_0,v_1,\dots ,v_{i-k}\}$.

\setcounter{subfigure}{0}
\begin{center}
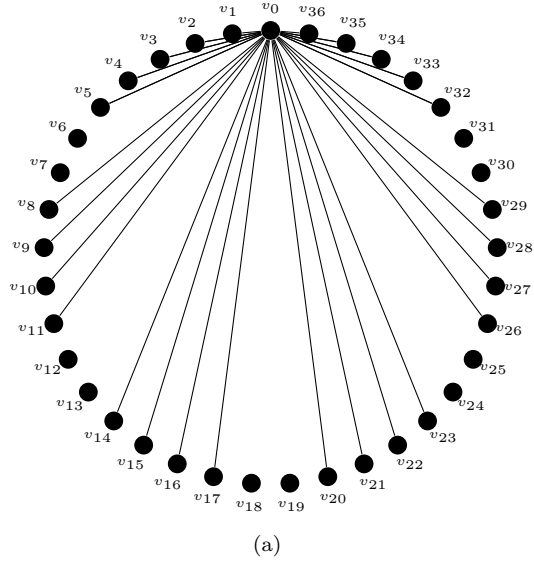
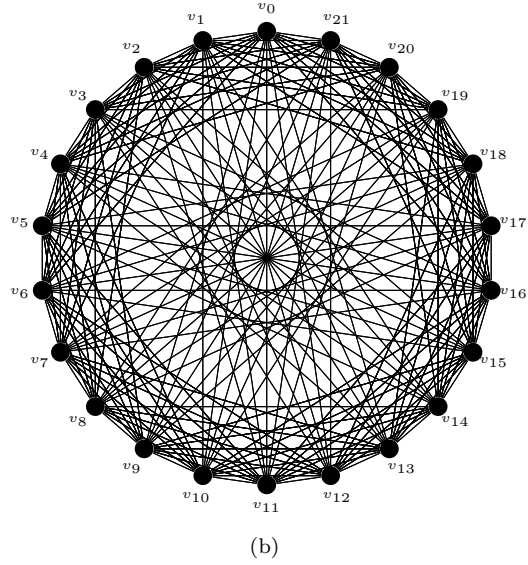
\begin{figure}[h!]
\def\c{3}
\centering
\qquad
\subfigure[]{\label{subfig:N(v0)}
\begin{tikzpicture}[scale=\c]
        \pgfmathsetmacro\k{6}
        \pgfmathsetmacro\q{6}
        
        \pgfmathsetmacro\n{int((\k)*\q+1)}
        \pgfmathsetmacro\nminusone{int(\n-1)}
        \pgfmathsetmacro\kminustwo{int(\k-2)}
        \pgfmathsetmacro\kminusone{int(\k-1)}
        \pgfmathsetmacro\qminusone{int(\q-1)}
        \GraphInit[vstyle=Classic]
    
        \foreach \i in {0,...,\nminusone} {
            \node[circle, fill=black, inner sep=2.5pt] (v\i) at ({360/\n * \i+90}:1) {};
            \node at ({360/\n * \i+90}:1.1) {\tiny $v_{\i}$};
        }
    
        \foreach \i in {0} {
    
            \foreach \j in {1,...,\kminusone} {
                \pgfmathtruncatemacro{\to}{Mod(\i+\j,\n)}
                \draw (v\i) -- (v\to);
            }
    
            \foreach \j in {1,...,\kminusone} {
                \pgfmathtruncatemacro{\to}{Mod(\i-\j,\n)}
                \draw (v\i) -- (v\to);
            }
    
            \foreach \j in {0,...,\qminusone} {
                \foreach \m in {2,...,\kminusone} {
                    \pgfmathtruncatemacro{\to}{Mod(\i + (\k)*\j + \m,\n)}
                    \draw (v\i) -- (v\to);
                }
            }
        }
    \end{tikzpicture}
}
\qquad
\subfigure[]{\label{subfig:G(3,7)}
\begin{tikzpicture}[scale=\c]
        \pgfmathsetmacro\k{7}
        \pgfmathsetmacro\q{3}
        
        \pgfmathsetmacro\n{int((\k)*\q+1)}
        \pgfmathsetmacro\nminusone{int(\n-1)}
        \pgfmathsetmacro\kminustwo{int(\k-2)}
        \pgfmathsetmacro\kminusone{int(\k-1)}
        \pgfmathsetmacro\qminusone{int(\q-1)}
        \GraphInit[vstyle=Classic]
    
        \foreach \i in {0,...,\nminusone} {
            \node[circle, fill=black, inner sep=2.5pt] (v\i) at ({360/\n * \i+90}:1) {};
            \node at ({360/\n * \i+90}:1.1) {\tiny $v_{\i}$};
        }
    
        \foreach \i in {0,...,\nminusone} {
    
            \foreach \j in {1,...,\kminusone} {
                \pgfmathtruncatemacro{\to}{Mod(\i+\j,\n)}
                \draw (v\i) -- (v\to);
            }
    
            \foreach \j in {1,...,\kminusone} {
                \pgfmathtruncatemacro{\to}{Mod(\i-\j,\n)}
                \draw (v\i) -- (v\to);
            }
    
            \foreach \j in {0,...,\qminusone} {
                \foreach \m in {2,...,\kminusone} {
                    \pgfmathtruncatemacro{\to}{Mod(\i + (\k)*\j + \m,\n)}
                    \draw (v\i) -- (v\to);
                }
            }
        }
    \end{tikzpicture}
}
\caption{The neighbourhood of $v_0$ in $G(6,6)$ (Figure~\ref{subfig:N(v0)}) and the graph $G(3,7)$ (Figure~\ref{subfig:G(3,7)}}\label{fig:Gqkneighbandfullgraph}.
\end{figure}
\end{center}

\begin{lemma}[\cite{CameronHoang2023P5C5}]
\label{lem:crtical}
    For all $k \geq 3$ and $q \geq 1$, $G(q,k)$ is $(k+1)$-vertex-critical.
\end{lemma}

For a given $q$ and $k$ and for $0\le i\le k-1$, let $V_i=\{v_t:t=\ i\pmod{k}\}$. It is clear that the $V_i$'s partition the vertex set of $\infcrit{q}{k}$. 

\begin{lemma}[\cite{CameronHoang2023P5C5}]
\label{lem:stable}
For $1\le i\le k$, $V_i$ is a stable set of $\infcrit{q}{k}$ and the only edge in $V_0$ is $v_0v_{qk}$.
\end{lemma}

\begin{lemma} \label{lem:cycleincomplement}
    Let $C = \{v_{0}, v_{i_1}, \ldots, v_{i_r}\}$ induce a cycle of length at least 4 in $\overline{G(n,k)}$. Then each remainder modulo $k$ must appear in the indices $ i_1, \ldots, i_{r}$.
\end{lemma}

\begin{proof}
    The neighbourhood of a vertex $v_i$ taken modulo $kq+1$ in $\overline{G(n,k)}$ is

    $$\{v_{i+kj+m} : m=0,1\text{ and } j=1,...,q-1\}.$$

    \noindent The remainder of the proof will consider the indices modulo $k$.
    Note that if $v_i \sim v_j$ and $j>i$ then $j \equiv i \text{ or } i+1 \pmod{k}$.
    Moreover, if $v_i \sim v_j$ and $j<i$ then $j \equiv i \text{ or } i-1 \pmod{k}$. Additionally it follows from Lemma \ref{lem:stable} that $v_i \sim v_j$ for all $j \equiv i \pmod{k}$ with the only exception being $v_0 \not\sim v_{kq}$.

    Let $C=\{v_0, v_{i_1}, \ldots, v_{i_r}\}$ induce a cycle with $r \geq 3$.
    Note $v_{i_1} \sim v_{0}$ and $v_{i_r} \sim v_{0}$ with $i_1, i_r > 0$. 
    Thus, $v_{i_1}, v_{i_r} \equiv 0 \text{ or } 1 \pmod{k}$.
    As $r \geq 3$ then $v_{i_1} \not\sim v_{i_r}$ and hence $v_{i_1} \not\equiv  v_{i_r} \pmod{k}$.
    Without loss of generality, let  $i_1 \equiv 1 \pmod{k}$ and $i_r \equiv 0 \pmod{k}$.

    We will now show by induction on $t$ that $i_t \equiv i_{t-1} \text{ or } i_{t-1}+1\pmod{k}$ for each $1 \leq t \leq r-1$.
    Clearly this is true for $t = 1$ so suppose this holds up to some $t \geq 1$.
    Note that if $i_t = kq$ then $i_t>i_{t-1}$ and $v_{i_t} \sim v_{i_{t-1}}$ so $i_t \equiv i_{t-1} \text{ or } i_{t-1}+1\pmod{k}$.
    So we may assume that $i_t < kq$.
    Recall that the neighbours of $v_{i_{t-1}}$ in $\overline{G(n,k)}$ have indices equivalent to $i_{t-1}-1$, $i_{t-1}$, or $i_{t-1}+1$ modulo $k$.
    So, to show a contradiction, suppose that $i_t \equiv i_{t-1}-1 \pmod{k}$.
    As $C$ is an induced cycle, $v_{i_t}$ is not adjacent to $v_{0}, v_{i_1}, \ldots, v_{i_{t-2}}$.
    Let $0 \leq \ell < k$ be such that $v_{i_{t-2}} \equiv \ell \pmod{k}$.
    Then by our inductive hypothesis $i_t \equiv \ell \text{ or } \ell +1 \pmod{k}$ and hence $i_t \equiv \ell-1 \text{ or } \ell \pmod{k}$.
    Additionally by our inductive hypothesis each remainder from $0, 1, \ldots, \ell$ appears in $v_{0}, v_{i_1}, \ldots, v_{i_{t-2}}$ and $\ell \geq 1$.
    Therefore $v_{i_t}$ is adjacent to at least one vertex in $v_{0}, v_{i_1}, \ldots, v_{i_{t-2}}$ which is a contradiction.
    
    Therefore $i_t \equiv i_{t-1} \text{ or } i_{t-1}+1\pmod{k}$ for each $1 \leq t \leq r-1$.
    As $v_{i_{r-1}} \sim v_{i_{r}}$ and $i_r \equiv 0 \pmod{k}$, we have $i_{r-1} \equiv k-1,0, \text{ or }1 \pmod{k}$.
    Note that $i_{r-1} \equiv 0 \text{ nor }1 \pmod{k}$ as otherwise $v_{i_{r-1}} \sim v_0$ and $C$ would not induce a cycle.
    Thus $i_{r-1} \equiv k-1 \pmod{k}$ and each non-zero remainder modulo $k$ must appear in $i_1, \ldots, i_{r-1}$ before $i_r \equiv 0 \pmod{k}$. 
\end{proof}

\begin{theorem}\label{thm:Gqkco-Ck-free}
    For all $k\ge 5$, $\numcrit{k}{\overline{C_4},\overline{C_5},\ldots,\overline{C_{k-1}}}=\infty$.
 \end{theorem}

\begin{proof}
From Lemma \ref{lem:crtical}, it suffices to show that the smallest hole in $\overline{G(n,k)}$ is $C_{k+1}$.
Without loss of generality, let $H=\{v_0, v_{i_1}, \ldots, v_{i_r}\}$ induce the smallest hole in $\overline{G(n,k)}$ with $r \geq 3$.
By Lemma \ref{lem:cycleincomplement} each remainder modulo $k$ must appear in $i_1, \ldots, i_{r}$.
Therefore $r\geq k$ and $|H| \geq k+1$.

\end{proof}

\begin{center}
\begin{figure}[htb]
\def\c{0.7}
\def\r{1.6}
\centering
\qquad
\subfigure[net]{
\scalebox{\c}{
\begin{tikzpicture}
\begin{scope}[every node/.style={circle,fill,draw}]
    \node (u1) at (0*\r,0*\r) {};
    \node (u2) at (1*\r,0*\r) {};
    \node (u3) at (0.5*\r,0.866025*\r) {};
    \node (u11) at (-0.6*\r,-0.866025*\r) {};
    \node (u22) at (1.6*\r,-0.866025*\r) {};
    \node (u33) at (0.5*\r,1.866025*\r) {};
\end{scope}
\begin{scope}
    \path [-] (u1) edge node {} (u2);    
    \path [-] (u2) edge node {} (u3);
    \path [-] (u3) edge node {} (u1);

    \path [-] (u1) edge node {} (u11);
    \path [-] (u2) edge node {} (u22);
    \path [-] (u3) edge node {} (u33);
\end{scope}
\end{tikzpicture}}
}
\qquad
\subfigure[co-net]{
\scalebox{\c}{
\begin{tikzpicture}
\begin{scope}[every node/.style={circle,fill,draw}]
    \node (u1) at (0*\r,0*\r) {};
    \node (u2) at (1*\r,0*\r) {};
    \node (u3) at (0.5*\r,-0.866025*\r) {};
    \node (u12) at (0.5*\r,0.866025*\r) {};
    \node (u23) at (1.5*\r,-0.866025*\r) {};
    \node (u31) at (-0.5*\r,-0.866025*\r) {};
\end{scope}
\begin{scope}
    \path [-] (u1) edge node {} (u2);    
    \path [-] (u2) edge node {} (u3);
    \path [-] (u3) edge node {} (u1);

    \path [-] (u1) edge node {} (u12);
    \path [-] (u2) edge node {} (u12);
    \path [-] (u2) edge node {} (u23);
    \path [-] (u3) edge node {} (u23);
    \path [-] (u3) edge node {} (u31);
    \path [-] (u1) edge node {} (u31);
\end{scope}
\end{tikzpicture}}
}
\caption{Other Graphs forbidden as induced subgraphs in $G(q,k)$}
\end{figure}
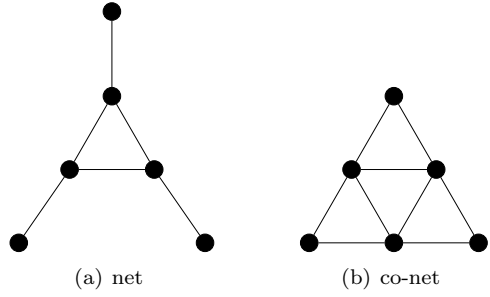
\end{center}

\begin{theorem}\label{thm:net-free}
    For all $k \geq 5$, $\numcrit{k}{P_5,\operatorname{net}}=\infty$.
\end{theorem}
\begin{proof}
    Let $G=\infcrit{q}{k}$ and suppose $G$ has an induced net, $N$. 
    By the symmetry of $G$, we may assume without loss of generality that $v_0$ is one of the leaves of $N$. 
    Let $v_{i}$ and $v_{j}$ be the two other leaves of $N$.
    By Definition~\ref{def:G(q,k)}, $i\equiv 0,1\pmod{k}$ and $j\equiv 0,1\pmod{k}$, but neither $i$ nor $j$ are $1$ or $qk$.
    Further, let $v_{t_0}$, $v_{t_i}$, and $v_{t_j}$ be the vertices of the triangle in $N$ such that $v_{t_h}\sim v_{h}$ for all $h\in\{0,i,j\}$. 
    We will consider two cases, and in both cases $i\equiv 0\pmod{k}$ (a justification that these cases are comprehensive comes at the end of this proof).

    For both cases, let $i=ki'$ for some  $i'\in \{1,\dots q-1\}$. 
    Since $v_{t_0}$ is adjacent to $v_0$ and nonadjacent to $v_i$, we must have that $t_0=kq$ or $kk'-1$ for some $0<k'<i'$ and $t_i=ki''+1$ for some $0<i''\le i'$ from Definition~\ref{def:G(q,k)}.\\

    \noindent\textit{Case 1:} $i\equiv 0\pmod{k}$ and $j\equiv 0 \pmod{k}$.
    
    Let $j=kj'$ for some $j'\in\{1,\dots q-1\}\setminus\{i'\}$. 
    Suppose also without loss of generality that $i'<j'$.
    Now, for any vertex $v_{ki''+1}$ for $0<i''\le i'$, we have $ki''+1 = kj' + k(q-j'+i''') + 2$.
    Therefore, by Definition~\ref{def:G(q,k)}, $v_j\sim v_{t_i}$, a contradiction.
    Thus, we cannot have both $i\equiv 0\pmod{k}$ and $j\equiv 0 \pmod{k}$.\\

    \noindent\textit{Case 2:} $i\equiv 0\pmod{k}$ and $j\equiv 1\pmod{k}$.
      
    Let $j=kj'+1$ for some $j'\in\{1,\cdots, q\}$. 
    Note that if $j'\le i'$, then $v_i\sim v_j$, a contradiction.
    So we must have $i'<j'$ and therefore $i<j$.
    Since $v_0\nsim v_{t_j}$ and $v_{j}\sim v_{t_j}$, we must have $t_j=kj''$ for $j'\le j''\le q-1$.
    Thus, $v_{t_j}\in V_0$, and $v_{t_j}\sim v_{t_0}$, so $t_0\neq qk$ from Lemma~\ref{lem:stable}.
    From above, we must now  have $t_0=kk'-1$ for some $0<k'<i'$.
    Since $v_{t_i}\sim v_{t_j}$, we must have $t_i < t_j$ (since $t_i\equiv 1\pmod{k}$, so is only adjacent to things with labels 0 mod $k$ for larger labels).
    Since $v_{t_j}\sim v_{t_0}$, we must have $t_j < t_0$.
    Further, we also have $t_0 < i$ from above, so we have
    $$t_i < t_j < t_0 < i < j$$
    or
    $$ki''+1 < kj'' < kk'-1 < ki' < kj'+1.$$
    But now, $v_{kj'+1}\sim v_{kk'-1}$, i.e., $v_j\sim v_{t_0}$, a contradiction.\\

    To complete the proof, we note that Cases 1 and 2 are comprehensive because if $i,j\equiv 1\pmod{k}$, then decrementing all vertex labels in $\infcrit{q}{k}$ by $i$ modulo $qk$ will result in $v_i$ being relabeled $v_0$, and $v_j$ being relabeled with something in $V_0$ and $v_0$ with something in $V_1$. 
    Thus, this case is equivalent to Case 2.
    
\end{proof}

\begin{theorem}
\label{thm:conet-free}
    For all $k \geq 5$, $\numcrit{k}{P_5,\operatorname{co-net}}=\infty$.
\end{theorem}
\begin{proof}
    We will prove this by showing the equivalent statement that $\overline{\infcrit{q}{k}}$ is net-free.
    As in the proof of Lemma~\ref{lem:cycleincomplement}, we note that he neighbourhood of vertex $v_i$ taken modulo $kq+1$ in $\overline{G(n,k)}$ is
    $$\{v_{i+kj+m} : m=0,1\text{ and } j=1,...,q-1\}.$$ 

    Suppose $\overline{\infcrit{q}{k}}$ has an induced net, $N$. 
    By the symmetry of $\overline{\infcrit{q}{k}}$, we may assume without loss of generality that $v_0$ is one of the vertices on the triangle in $N$. 
    Let $v_{i}$ and $v_{j}$ be the two other vertices on the triangle in $N$.
    Since $v_i\sim v_0$ and $v_j\sim v_0$, we must have $i\equiv 0,1\pmod{k}$ and $j\equiv 0,1\pmod{k}$, but neither $i$ nor $j$ are $1$ or $qk$.
    Further, let $v_{\ell_0}$, $v_{\ell_i}$, and $v_{\ell_j}$ be the leaves of $N$ such that $v_{\ell_h}\sim v_{h}$ for all $h\in\{0,i,j\}$. 
    
    As with the proof of Theorem~\ref{thm:net-free}, we will consider two cases, both of which have $i\equiv 0\pmod{k}$.
    Let $i=ki'$ for some $1\le i'\le q-1$.
    Since $v_{\ell_0}\sim v_0$ and $v_{\ell_0}\nsim v_i$, we must have $\ell_0=kk'+1$ for some $1\le k'\le i'$.
    Similarly, since $v_{\ell_i}\sim v_i$ and $v_{\ell_i}\nsim v_0$, we must have $\ell_i=qk$ or $\ell_i=ki''-1$ for some $1\le i''<i'$.\\

    \noindent\textit{Case 1:} $i\equiv 0\pmod{k}$ and $j\equiv 0\pmod{k}$.

    Let $j=kj'$ for some $1\le j'\le q-1$.
    Suppose also without loss of generality that $i'<j'$.
    But now we cannot have $\ell_i=qk$ or $\ell_i=ki''-1$ for some $1\le i''<i'$ as then $v_j\sim v_{\ell_i}$, a contradiction.\\

    \noindent\textit{Case 2:} $i\equiv 0\pmod{k}$ and $j\equiv 1\pmod{k}$.
    
    Thus, $v_j\in V_1$, and from Lemma~\ref{lem:stable} applied to $\overline{\infcrit{q}{k}}$, we have $v_j\sim v_{\ell_0}$, a contradiction.\\

    As with the proof of Theorem~\ref{thm:net-free}, we note that Cases 1 and 2 are comprehensive, because if $i,j\equiv 1\pmod{k}$, then decrementing all vertex labels in $\infcrit{q}{k}$ by $i$ modulo $qk$ will result in $v_i$ being relabeled $v_0$, and $v_j$ being relabeled with something in $V_0$ and $v_0$ with something in $V_1$. 
    Thus, this case is equivalent to Case 2.
    This completes the proof. 
    
\end{proof}

\section{Conclusion}\label{sec:conclusion}

In this paper we showed there are only finitely many $k$-vertex critical graphs for large sub-families of $P_5$-free and $(P_4+\ell P_1)$-free graphs.
Most notably, Theorem~\ref{thm:P5Bipartite} gives a complete dichotomy on the finiteness of $k$-vertex critical $(P_5, H)$-free graphs when $H$ is bipartite. This resolves two of the remaining cases of Problem~\ref{prob:P5Hord5} and also leads us to Problem~\ref{prob:P5Hlargerchrom} when $\chi(H) \geq 3$.
Our results generalize many previous ones in the literature.
Theorem~\ref{thm:P5Bipartite} generalizes the results that $\numcrit{k}{P_5,H}<\infty$ for all $k$ when $H$ is complete bipartite~\cite{Kaminski2019}, banner~\cite{Brause2022}, $\overline{P_3+P_2}$~\cite{CaiGoedgebeurHuang2023}, $K_{1,3}+P_1$~\cite{xia2024results}, or chair~\cite{Jooken2026}.
Theorem~\ref{thm:P5co-(Kell+P3)-free} generalizes the results that $\numcrit{k}{P_5,H}<\infty$ for all $k$ when $H$ is: $\overline{P_3+P_2}$~\cite{CaiGoedgebeurHuang2023}, $\overline{K_3+2P_1}$~\cite{xia2024results}, or cricket~\cite{Jooken2026}.
Note that Theorems~\ref{thm:P5Bipartite} and~\ref{thm:P5co-(Kell+P3)-free} also generalize the results that $\numcrit{k}{2P_2,(m,\ell)\text{-squid}}<\infty$ for all $k,\ell$~\cite{Adekanye2024IWOCA} for $m=4$ and $m=3$, respectively,  where $(m,\ell)\text{-squid}$ is the graph obtained from $C_m$ by attaching $\ell$ leaves to one of its vertices.

We also gave a deeper structural analysis of the $(2P_2,K_3+P_1)$-free graphs $G(q,k)$ from~\cite{Hoang2015,CameronHoang2023P5C5} which we showed offers an infinite family of $(k+1)$-vertex-critical $P_5$-free graphs which are also: 

\begin{itemize}
 \item $(\overline{C_4},\ldots,\overline{C_{k}})$-free when $k\ge 4$ (by Theorem~\ref{thm:Gqkco-Ck-free})
 \item net-free when $k\geq 1$ (by Theorem~\ref{thm:net-free})
 \item co-net-free when $k\geq 1$ (by Theorem~\ref{thm:conet-free})
\end{itemize}

\noindent It is clear that any new graphs $H$ that satisfy Problem~\ref{prob:P5Hlargerchrom} must contain an induced odd-cycle. Since the graphs in question are $P_5$-free, this only leaves odd cycle lengths of $3$ and $5$. 
If $H$ contains an induced $C_5$, then it follows from~\cite{CameronHoang2023P5C5} that $\numcrit{k}{P_5, H}=\infty$ for all $k\ge 6$. Thus, any graph $H$ satisfying Problem \ref{prob:P5Hlargerchrom} must be $C_5$-free, contain $K_3$, and also be free of all of the graphs listed in the previous paragraph.
From this, and computational evidence, we pose the following conjecture graphs $H$ with $\chi(H) =3$.

\begin{conjecture}
    Let $H$ be a graph with $\chi(H)=3$. 
    Then $\numcrit{k}{P_5,H}<\infty$ for all $k\ge 5$ if and only if $H$ is $(2P_2,K_3+P_1,C_5,\overline{C_6},\operatorname{net},\operatorname{co-net})$-free.
\end{conjecture}

Note that $\numcrit{5}{P_5, C_5}<\infty$~\cite{Hoang2015}, and thus it is open to determine the finiteness of $\numcrit{k}{P_5, H}$ when $H$ contains an induced $C_5$ which we now pose.

\begin{problem}
    \label{prob:C5andkis5} For which graphs $H$ is $\numcrit{5}{P_5, H}<\infty$ when $H$ contains an induced $C_5$?
\end{problem}

As a first step toward resolving Problem~\ref{prob:C5andkis5}, we now consider each graph $H$ that is $(2P_2,K_3+P_1)$-free, contains an induced $C_5$, and has order $6$ (note that this leaves only the four graphs in Figure~\ref{fig:order6C5containinggraphs}).
We begin with an observation that will solve one of these cases and that answers a question that astute readers may have about where the induced half-graphs are in $\infcrit{q}{k}$.

\begin{center}
\begin{figure}[htb]
\def\c{0.3}
\def\r{1.5}
\centering
\qquad
\subfigure[$C_5+P_1$]{
    \scalebox{\c}{
        \begin{tikzpicture}
        \GraphInit[vstyle=Classic]
        \Vertex[L=\hbox{},x=0.0cm,y=3.1235cm]{v0}
        \Vertex[L=\hbox{},x=5.0cm,y=3.0618cm]{v1}
        \Vertex[L=\hbox{},x=0.942cm,y=0.0cm]{v2}
        \Vertex[L=\hbox{},x=2.5253cm,y=5.0cm]{v3}
        \Vertex[L=\hbox{},x=4.0006cm,y=0.0cm]{v4}
        \Vertex[L=\hbox{},x=2.5253cm,y=2.5cm]{v5}
        
        \Edge[](v0)(v2)
        \Edge[](v0)(v3)
        \Edge[](v1)(v3)
        \Edge[](v1)(v4)
        \Edge[](v2)(v4)
        \end{tikzpicture}
    }\label{subfig:C5+P1}
}
\qquad
\subfigure[$W_5$]{
    \scalebox{\c}{
        \begin{tikzpicture}
        \GraphInit[vstyle=Classic]
        \Vertex[L=\hbox{},x=0.0cm,y=3.1235cm]{v0}
        \Vertex[L=\hbox{},x=5.0cm,y=3.0618cm]{v1}
        \Vertex[L=\hbox{},x=0.942cm,y=0.0cm]{v2}
        \Vertex[L=\hbox{},x=2.5253cm,y=5.0cm]{v3}
        \Vertex[L=\hbox{},x=4.0006cm,y=0.0cm]{v4}
        \Vertex[L=\hbox{},x=2.5253cm,y=2.5cm]{v5}
        
        \Edge[](v0)(v2)
        \Edge[](v0)(v3)
        \Edge[](v1)(v3)
        \Edge[](v1)(v4)
        \Edge[](v2)(v4)
        \Edge[](v0)(v5)
        \Edge[](v5)(v3)
        \Edge[](v1)(v5)
        \Edge[](v5)(v2)
        \Edge[](v5)(v4)
        \end{tikzpicture}
    }\label{subfig:W5}
}
\qquad
\subfigure[twin-$C_5$]{
    \scalebox{\c}{
        \begin{tikzpicture}
        \GraphInit[vstyle=Classic]
        \Vertex[L=\hbox{},x=0.0cm,y=3.1235cm]{v0}
        \Vertex[L=\hbox{},x=5.0cm,y=3.0618cm]{v1}
        \Vertex[L=\hbox{},x=0.942cm,y=0.0cm]{v2}
        \Vertex[L=\hbox{},x=2.5253cm,y=5.0cm]{v3}
        \Vertex[L=\hbox{},x=4.0006cm,y=0.0cm]{v4}
        \Vertex[L=\hbox{},x=2.5253cm,y=2.5cm]{v5}
        
        \Edge[](v0)(v2)
        \Edge[](v0)(v3)
        \Edge[](v1)(v3)
        \Edge[](v1)(v4)
        \Edge[](v2)(v4)
        \Edge[](v0)(v5)
        \Edge[](v1)(v5)
        \end{tikzpicture}
    }\label{subfig:twin-C5}
}
\qquad
\subfigure[$\overline{X_{37}}$]{
    \scalebox{\c}{
        \begin{tikzpicture}
        \GraphInit[vstyle=Classic]
        \Vertex[L=\hbox{},x=0.0cm,y=3.1235cm]{v0}
        \Vertex[L=\hbox{},x=5.0cm,y=3.0618cm]{v1}
        \Vertex[L=\hbox{},x=0.942cm,y=0.0cm]{v2}
        \Vertex[L=\hbox{},x=2.5253cm,y=5.0cm]{v3}
        \Vertex[L=\hbox{},x=4.0006cm,y=0.0cm]{v4}
        \Vertex[L=\hbox{},x=2.5253cm,y=2.5cm]{v5}
        
        \Edge[](v0)(v2)
        \Edge[](v0)(v3)
        \Edge[](v1)(v3)
        \Edge[](v1)(v4)
        \Edge[](v2)(v4)
        \Edge[](v2)(v5)
        \Edge[](v3)(v5)
        \Edge[](v4)(v5)
        \end{tikzpicture}
    }\label{subfig:co-X37}
}
\caption{All $(2P_2,K_3+P_1)$-free graphs containing an induced $C_5$ of order $6$.}\label{fig:order6C5containinggraphs}
\end{figure}
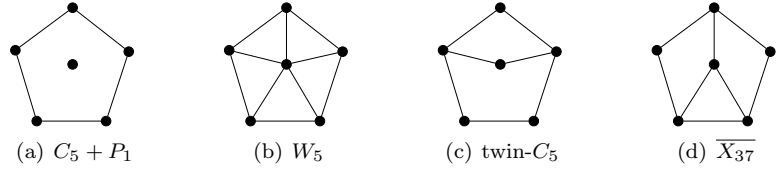
\end{center}

\begin{observation} \label{obs:nonneighbourhoodhalfgraph}
    For all $q,k \geq 1$, the graph induced by the non-neighbours of $v_0$ in $G(q,k)$ is the half-graph $H_{q-1}$. 
    

\end{observation}

\begin{proof}

Let $A=V_1 \setminus \{v_1\}$ and $B=V_0 \setminus \{v_0, v_{kq}\}$. 
From Definition \ref{def:G(q,k)}, the vertices in the non-neighbourhood of $v_0$ in $G(q,k)$ are $A \cup B$.
Moreover, from Lemma \ref{lem:stable}, $A$ and $B$ are each stable sets with $|A|=|B|=q-1$. 
Now consider vertices $v_{ik} \in B$ and $v_{jk+1} \in A$.
By Definition \ref{def:G(q,k)}, $v_{ik} \sim v_{jk+1}$ if and only if $j \geq i$.
Thus the non-neighbourhood of $v_0$ in $G(q,k)$ is isomorphic to to the half-graph $H_{q-1}$ where $v_{ik} \in B$ and $v_{jk+1} \in A$ correspond to $b_i$ and $a_i$ in $H_{q-1}$ respectively.
\end{proof}

\begin{theorem}
    For all $k \geq 1$ and $q \geq 1$, $G(q,k)$ is ($C_5+P_1$)-free
\end{theorem}
\begin{proof}
    The proof follows immediately from Observation~\ref{obs:nonneighbourhoodhalfgraph}.
\end{proof}

\begin{theorem}
    For all $k \geq 1$ and $q \geq 1$, $G(q,k)$ is $W_5$-free.
\end{theorem}

\begin{proof}
    It suffices to show $\overline{G(q,k)}$ is $C_5+P_1$ free. 
    Suppose there is an induced $C_5$ in $\overline{G(q,k)}$. 
    Without loss of generality let $C = \{v_0, v_{i_1}, \ldots , v_{i_4}\}$ induce the $C_5$. 
    From Lemma~\ref{lem:cycleincomplement}, every remainder modulo $k$ must appear in $v_{i_1}, \ldots , v_{i_4}$.
    Since each $V_k$ is a clique, then every vertex has a neighbour on the cycle expect possibly $v_{kq}$.
    However there are two vertices in $C$ with indices congruent to zero modulo $k$ and hence $v_{kq}$ is adjacent to at least one of them.
\end{proof}

Using Jooken's program~\cite{JorikVertexCriticalGenerator} (originally developed for~\cite{Xiaetal2023}) for exhaustively generating all $k$-vertex-critical $(P_{t},H)$-free graphs, we find that 
$$\numcrit{5}{P_5,\text{twin-}C_5}=287,\text{ and}$$ 
$$\numcrit{5}{P_5,\overline{X_{37}}}=188.$$
Both of these sets of $5$-vertex-critical graphs are available at~\cite{CameronP5HfreeGitHub} in graph6 format.
We note Jooken's program is an independent implementation of a similar program developed in~\cite{GoedgebeurSchaudt2018}, itself an optimized version of the original developed in~\cite{Hoang2015}.

We conclude this section by noting the potential for Theorem~\ref{thm:P5Hn} to provide a new proof technique of building around an arbitrarily large induced half-graph for showing $\numcrit{k}{P_5,H}<\infty$ for new graphs $H$. 
This works as we now know that any infinite family of $(P_5,H)$-free graphs necessarily have infinitely many graphs that contain $H_n$ for every fixed natural number $n$.
We also note this potential for a new proof technique of building around an arbitrarily large $\overline{L(K_{2,n})}$  toward a solution to Conjecture~\ref{conj:P4UellP1finite} from Theorem~\ref{thm:P4ellP1Cocktail}.
A similar idea for building around an induced $P_3+\ell P_1$ for arbitrarily large $\ell$ (which we note is contained in large half-graphs) has been successfully used to bound the number of $k$-vertex-critical graphs in various graph families~\cite{Adekanye2024IWOCA,BeatonCameron2025cogemfreeord4finite,BeatonCameron2026IWOCA,Jooken2026}.

\subsection*{Acknowledgments} 
The research of Iain Beaton was supported by the Natural Sciences and Engineering Research Council of Canada (NSERC) grants RGPIN-2025-06012 and DGECR-2025-00001.
The research of Ben Cameron was also supported by NSERC, grants RGPIN-2022-03697 and DGECR-2022-00446. 

\bibliographystyle{abbrv}
\bibliography{refs}
\end{document}